\documentclass[oneside,british]{amsart}
\usepackage[T1]{fontenc}
\usepackage[latin9]{inputenc}
\usepackage{amsbsy}
\usepackage{amstext}
\usepackage{amsthm}
\usepackage{amssymb}

\makeatletter
\numberwithin{equation}{section}
\numberwithin{figure}{section}
\theoremstyle{plain}
\newtheorem*{thm*}{\protect\theoremname}
\theoremstyle{plain}
\newtheorem*{cor*}{\protect\corollaryname}
\theoremstyle{plain}
\newtheorem{thm}{\protect\theoremname}[section]
\theoremstyle{plain}
\newtheorem{cor}[thm]{\protect\corollaryname}
\theoremstyle{remark}
\newtheorem{rem}[thm]{\protect\remarkname}
\theoremstyle{plain}
\newtheorem{lem}[thm]{\protect\lemmaname}
\theoremstyle{plain}
\newtheorem{prop}[thm]{\protect\propositionname}

\makeatother

\usepackage{babel}
\providecommand{\corollaryname}{Corollary}
\providecommand{\lemmaname}{Lemma}
\providecommand{\propositionname}{Proposition}
\providecommand{\remarkname}{Remark}
\providecommand{\theoremname}{Theorem}

\begin{document}

\title{On self-similar measures with absolutely continuous projections and
dimension conservation in each direction}

\author{\noindent Ariel Rapaport}

\date{\noindent September 26, 2018}


\begin{abstract}
Relying on results due to Shmerkin and Solomyak, we show that outside
a $0$-dimensional set of parameters, for every planar homogeneous
self-similar measure $\nu$, with strong separation, dense rotations
and dimension greater than $1$, there exists $q>1$ such that $\{P_{z}\nu\}_{z\in S}\subset L^{q}(\mathbb{R})$.
Here $S$ is the unit circle and $P_{z}w=\left\langle z,w\right\rangle $
for $w\in\mathbb{R}^{2}$. We then study such measures. For instance,
we show that $\nu$ is dimension conserving in each direction and
that the map $z\rightarrow P_{z}\nu$ is continuous with respect to
the weak topology of $L^{q}(\mathbb{R})$.
\end{abstract}

\maketitle

\section{\label{sec:Introduction}Introduction}

Denote by $\mathbb{D}$ the open unit disc in $\mathbb{C}$ and let
$0\ne\lambda\in\mathbb{D}$ be with $\arg\lambda\notin\pi\mathbb{Q}$.
Consider a homogeneous IFS on $\mathbb{C}$
\begin{equation}
\{\varphi_{i}(w)=\lambda w+a_{i}\}_{i\in\Lambda},\label{eq:IFS intro}
\end{equation}
with the strong separation condition (SSC), and a self-similar measure
\begin{equation}
\nu=\sum_{i\in\Lambda}p_{i}\cdot\varphi_{i}\nu\:.\label{eq:SSM intro}
\end{equation}
It is among the most basic planar self-similar measures. Hence it
is a natural question in fractal geometry to study the dimension and
continuity of the projections $\{P_{z}\nu\}_{z\in S}$ and slices
\[
\{\{\nu_{z,w}\}_{w\in\mathbb{C}}\::\:z\in S\}\:.
\]
Here $S$ is the unit circle of $\mathbb{C}$, $P_{z}w=\mathrm{Re}(z\cdot\overline{w})$
for $w\in\mathbb{C}$, and $\{\nu_{z,w}\}_{w\in\mathbb{C}}$ is the
disintegration of $\nu$ with respect to $P_{z}^{-1}(\mathcal{B})$,
where $\mathcal{B}$ is the Borel $\sigma$-algebra.

Dimensionwise, the behaviour of the projections is as regular as possible.
Indeed, Hochman and Shmerkin \cite{HS} have proven that $P_{z}\nu$
is exact dimensional, with
\[
\dim P_{z}\nu=\min\{1,\dim\nu\},
\]
for each $z\in S$. A version of this, for self-similar sets with
dense rotations, was first proven by Peres and Shmerkin \cite{PS}.
Considering the absolute continuity of the projections, Shmerkin and
Solomyak \cite{SS1} have shown, assuming $\dim\nu>1$, that the set
\[
E=\{z\in S\::\:P_{z}\nu\mbox{ is singular}\}
\]
has zero Hausdorff dimension.

Let us turn to discuss the concept of dimension conservation and the
dimension of slices. A Borel probability measure $\mu$ on $\mathbb{C}$
is said to be dimension conserving (DC), with respect to the projection
$P_{z}$, if
\[
\dim_{H}\mu=\dim_{H}P_{z}\mu+\dim_{H}\mu_{z,w}\mbox{ for \ensuremath{\mu}-a.e. \ensuremath{w\in\mathbb{C}}},
\]
where $\dim_{H}$ stands for Hausdorff dimension. It always holds
that $\mu$ is DC with respect to $P_{z}$ for almost every $z\in S$.
This follows from results, valid for general measures, regarding the
typical dimension of projections (see \cite{HK}) and slices (see
\cite{JM}). Falconer and Jin \cite{FJ1} have shown that $\mu$ is
DC, with respect to $P_{z}$ for all $z\in S$, whenever $\mu$ is
self-similar with a finite rotation group. An analogues statement,
for self-similar sets with the SSC, was first proven by Furstenberg
\cite{Fur}. Another related result for sets is due to Falconer and
Jin \cite{FJ2}. They showed that if $K\subset\mathbb{C}$ is self-similar,
with $\dim K>1$ and a dense rotation group, then for every $\epsilon>0$
there exists $N_{\epsilon}\subset S$, with $\dim_{H}N_{\epsilon}=0$,
such that for $z\in S\setminus N_{\epsilon}$ the set
\[
\{x\in\mathbb{R}\::\:\dim_{H}(K\cap P_{z}^{-1}\{x\})>\dim K-1-\epsilon\}
\]
has positive length.

Taking these results into account, one might expect the sets $E$,
defined above, and
\[
F=\{z\in S\::\:\nu\mbox{ is not DC with respect to \ensuremath{P_{z}}}\}
\]
to be empty whenever the dimension of $\nu$ exceeds $1$. Nevertheless,
in \cite{Rap2} the author has constructed an example, of a measure
$\nu$ as above, for which $E$ and $F$ are nonempty and even residual.
Moreover, in this example there exists a $G_{\delta}$-subset of directions
$z\in S$ for which $\nu_{z,w}$ is discrete, and hence has dimension
$0$, for $\nu$-a.e. $w\in\mathbb{C}$.

In this paper we shall show that, despite of the last result, it is
typically the case that $E$ and $F$ are empty. More precisely, we
shall prove the following theorem. A sharper version of it is stated
in Section \ref{sec:Statement-of-results}.
\begin{thm*}
There exists a set $\mathcal{E}\subset\mathbb{D}$, with $\dim_{H}\mathcal{E}=0$,
such that the following holds. Let $\lambda\in\mathbb{D}\setminus\mathcal{E}$
be with $\arg\lambda\notin\pi\mathbb{Q}$, let $\nu$ be a self-similar
measure with the SSC as in (\ref{eq:SSM intro}), and assume that
$\dim_{H}\nu>1$. Then there exists $1<q<\infty$ such that,
\begin{enumerate}
\item \label{enu:part 1}$P_{z}\nu\in L^{q}(\mathbb{R})$ for each $z\in S$;
\item \label{enu:part 2}the map which takes $z\in S$ to $P_{z}\nu\in L^{q}(\mathbb{R})$
is continuous with respect to the weak topology of $L^{q}(\mathbb{R})$;
\item \label{enu:part 3}for each $z\in S$ the measure $\nu_{z,w}$ has
exact dimension $\dim_{H}\nu-1$ for $\nu$-a.e. $w\in\mathbb{C}$.
\end{enumerate}
\end{thm*}
Note that by parts (\ref{enu:part 1}) and (\ref{enu:part 3}) it
follows that $\nu$ is DC, with respect to $P_{z}$, for all $z\in S$.

Part (\ref{enu:part 1}) follows almost directly from results which
are due to Shmerkin and Solomyak \cite{SS1} and \cite{SS2} and Shmerkin
\cite{Sh1} and \cite{Sh2}. Our main contribution is the derivation
of parts (\ref{enu:part 2}) and (\ref{enu:part 3}). For this we
first show that, assuming part (\ref{enu:part 1}), the collection
$\{P_{z}\nu\}_{z\in S}$ is bounded in $(L^{q}(\mathbb{R}),\Vert\cdot\Vert_{q})$.
The reflexivity of $L^{q}(\mathbb{R})$, for $1<q<\infty$, plays
an important role here.

We obtain a much stronger form of continuity in (\ref{enu:part 2})
if instead of $\dim_{H}\nu>1$ we assume that the correlation dimension
of $\nu$ exceeds $1$. More precisely, in Theorem \ref{thm:continuous frac deriv}
it is shown that, under this stronger assumption, there exits $\gamma>0$
so that $P_{z}\nu$ lies in the $2,\gamma$-Sobolev space for each
$z\in S$ and that the map $z\rightarrow P_{z}\nu$ is continuous
with respect to the corresponding Sobolev norm.

The boundedness of $\{P_{z}\nu\}_{z\in S}$ in $L^{q}(\mathbb{R})$
gives also a result regarding the measure class of the projections.
Denote by $K$ the self-similar set corresponding to the IFS (\ref{eq:IFS intro}).
By extending an argument given in \cite{MS} and \cite{PSS}, we show
that, under the assumptions of the theorem above, the measures $P_{z}\nu$
and $\mathcal{L}|_{P_{z}(K)}$ are equivalent for all $z\in S$. Here
$\mathcal{L}$ is the Lebesgue measure on $\mathbb{R}$.

Finally, our results also apply to self-similar sets. We recall the
following definition due to Furstenberg \cite{Fur}. A subset $A\subset\mathbb{C}$
is said to be dimension conserving (DC), with respect to the projection
$P_{z}$, if for some $\delta\ge0$ 
\[
\delta+\dim_{H}\{x\in\mathbb{R}\::\:\dim_{H}\left(A\cap P_{z}^{-1}\{x\}\right)\ge\delta\}\ge\dim_{H}A,
\]
where the dimension of the empty set is defined to be $-\infty$.
From our results on measures we shall obtain the following corollary.
A sharper version of it is given in Section \ref{sec:Statement-of-results}.
Note its close connection with the priorly mentioned result from \cite{FJ2}.
\begin{cor*}
Let $\mathcal{E}\subset\mathbb{D}$ be the $0$-dimensional set from
the theorem above. Let $\lambda\in\mathbb{D}\setminus\mathcal{E}$
be with $\arg\lambda\notin\pi\mathbb{Q}$, let $K$ be a self-similar
set with the SSC corresponding to an IFS as in (\ref{eq:IFS intro}),
and assume that $s=\dim_{H}K>1$. Then there exists $c>0$ such that
for every $z\in S$,
\[
\mathcal{L}\{x\in\mathbb{R}\::\:\dim_{H}\left(K\cap P_{z}^{-1}\{x\}\right)=s-1\}>c\:.
\]
In particular, $K$ is DC with respect to $P_{z}$ for all $z\in S$. 
\end{cor*}
The rest of this paper is organized as follows. In Section \ref{sec:Preliminaries}
we give the necessary definitions. In Section \ref{sec:Statement-of-results}
we state our results. In Section \ref{sec:Proofs-of-parts 1,2} we
show that the projections $P_{z}\nu$ all belong to the appropriate
function space, and establish the continuity of the map $z\rightarrow P_{z}\nu$.
In Section \ref{sec:Proof-of-dc} we prove the result regarding the
slices of measures. In Section \ref{sec:Proof-of-Cor for sets} we
establish our results for self-similar sets. Finally, in Section \ref{sec:Proof-of equiv measures}
we prove the statement regarding the measure class of the projections.$\newline$

\textbf{Acknowledgment.} I would like to thank P. Shmerkin for suggesting
to consider a general $q>1$ in Theorem \ref{thm:main_thm} instead
of just $q=2$. I would also like to thank M. Hochman for helpful
discussions.

\section{\label{sec:Preliminaries}Preliminaries}

\subsection{Self-similar sets and measures}

Write $\mathbb{D}=\{z\in\mathbb{C}\::\:|z|<1\}$ and $S=\{z\in\mathbb{C}\::\:|z|=1\}$.
Given a metric space $X$, which will always be $\mathbb{C}$ or $\mathbb{R}$,
denote by $\mathcal{M}(X)$ the collection of all compactly supported
Borel probability measures on $X$. For a finite index set $\Lambda$
write
\[
\mathbb{P}_{\Lambda}=\{(p_{i})_{i\in\Lambda}\in(0,1)^{\Lambda}\::\:\sum_{i\in\Lambda}p_{i}=1\}\:.
\]
Given $p\in\mathbb{P}_{\Lambda}$, $(a_{i})_{i\in\Lambda}=\mathbf{a}\in\mathbb{C}^{\Lambda}$,
and $\lambda\in\mathbb{D}$, let $\nu_{\lambda,\mathbf{a}}^{p}\in\mathcal{M}(\mathbb{C})$
be the self-similar measure corresponding to the IFS
\begin{equation}
\mathcal{F}_{\lambda,\mathbf{a}}=\{f_{\lambda,\mathbf{a}}^{i}(z)=\lambda z+a_{i}\}_{i\in\Lambda}\label{eq:def of IFS}
\end{equation}
and $p$, i.e. $\nu_{\lambda,\mathbf{a}}^{p}$ is the unique member
of $\mathcal{M}(\mathbb{C})$ such that
\[
\nu_{\lambda,\mathbf{a}}^{p}=\sum_{i\in\Lambda}p_{i}\cdot f_{\lambda,\mathbf{a}}^{i}\nu_{\lambda,\mathbf{a}}^{p}\:.
\]
Denote by $K_{\lambda,\mathbf{a}}\subset\mathbb{C}$ the attractor
of $\mathcal{F}_{\lambda,\mathbf{a}}$, i.e. $K_{\lambda,\mathbf{a}}$
is the unique nonempty compact subset of $\mathbb{C}$ with
\begin{equation}
K_{\lambda,\mathbf{a}}=\cup_{i\in\Lambda}f_{\lambda,\mathbf{a}}^{i}(K_{\lambda,\mathbf{a}})\:.\label{eq:def_of_attractor}
\end{equation}
Since $p>0$ it holds that $K_{\lambda,\mathbf{a}}=supp(\nu_{\lambda,\mathbf{a}}^{p})$.
We say that the strong separation condition (SSC) holds if the union
in (\ref{eq:def_of_attractor}) is disjoint.

\subsection{Projections and disintegrations}

For $z,w\in\mathbb{C}$ write $\left\langle z,w\right\rangle =\mathrm{Re}(z\cdot\overline{w})$,
then $\left\langle \cdot,\cdot\right\rangle $ is the standard inner
product on $\mathbb{C}$ when it is identified with $\mathbb{R}^{2}$.
For $z\in S$ let $P_{z}:\mathbb{C}\rightarrow\mathbb{R}$ be with
$P_{z}w=\left\langle z,w\right\rangle $ for $w\in\mathbb{C}$. Note
that given $\mu\in\mathcal{M}(\mathbb{C})$ the measure $P_{z}\mu$
is, up to affine equivalence, the pushforward of $\mu$ by the orthogonal
projection onto the line $z\cdot\mathbb{R}$.

Denote by $\mathcal{B}$ the Borel $\sigma$-algebra of $\mathbb{R}$
or $\mathbb{C}$. For $\mu\in\mathcal{M}(\mathbb{C})$ and $z\in S$
let $\{\mu_{z,w}\}_{w\in\mathbb{C}}\subset\mathcal{M}(\mathbb{C})$
be the disintegration of $\mu$ with respect to $P_{z}^{-1}(\mathcal{B})$,
as defined in Theorem 5.14 in \cite{EW}. This means that $\mu_{z,w}$
is supported on $P_{z}^{-1}(P_{z}w)$ for $\mu$-a.e. $w\in\mathbb{C}$,
and that for each bounded \emph{$\mathcal{B}$-}measurable $f:\mathbb{C}\rightarrow\mathbb{R}$
\[
\int f\:d\mu_{z,w}=E_{\mu}(f\mid P_{z}^{-1}(\mathcal{B}))(w)\mbox{ for \ensuremath{\mu}-a.e. \ensuremath{w\in\mathbb{C}}\:.}
\]
In the last equality, the right hand side is the conditional expectation
of $f$ given $P_{z}^{-1}(\mathcal{B})$ with respect to $\mu$.

\subsection{\label{subsec:Dimension-of-measures}Dimension of measures}

Let $X$ be $\mathbb{C}$ or $\mathbb{R}$. For $E\subset X$ denote
by $\dim_{H}E$ the Hausdorff dimension of $E$. Given $\mu\in\mathcal{M}(X)$
write $\dim_{H}\mu$ for the Hausdorff dimension of $\mu$, which
is defined by
\[
\dim_{H}\mu=\inf\{\dim_{H}E\::\:E\subset X\text{ is a Borel set with }\mu(E)>0\}\:.
\]
For $x\in X$ and $\delta>0$ let $B(x,\delta)$ be the closed ball
in $X$ with centre $x$ and radius $\delta$. It is said that $\mu$
has exact dimension $s\ge0$ if
\[
\underset{\delta\downarrow0}{\lim}\:\frac{\log\mu(B(x,\delta))}{\log\delta}=s\:\text{ for \ensuremath{\mu}-a.e. \ensuremath{x\in X}}\:.
\]
It is well known that in this case $s=\dim_{H}\mu$ (see \cite[Chapter 10]{Fal2}).

For $1<q<\infty$ denote by $D_{q}(\mu)$ the lower $L^{q}$ dimension
of $\mu$, which is defined by
\[
D_{q}(\mu)=\underset{\delta\downarrow0}{\liminf}\:\frac{\log\int\left(\mu\left(B\left(x,\delta\right)\right)\right)^{q-1}\:d\mu(x)}{(q-1)\log\delta}\:.
\]
It always holds that $\dim_{H}\mu\ge D_{q}(\mu)$ and $D_{q'}(\mu)\ge D_{q}(\mu)$
for all $1<q'<q$. A proof of these standard facts can be found in
\cite{FLR}. The number $D_{2}(\mu)$ is called the lower correlation
dimension of $\mu$.

For $p\in\mathbb{P}_{\Lambda}$ write $\Vert p\Vert_{q}^{q}=\sum_{i\in\Lambda}p_{i}^{q}$.
It is not hard to verify that for $p\in\mathbb{P}_{\Lambda}$, $\mathbf{a}\in\mathbb{C}^{\Lambda}$,
and $\lambda\in\mathbb{D}$, such that the SSC is satisfied,
\begin{equation}
D_{q}(\nu_{\lambda,\mathbf{a}}^{p})=\frac{\log\Vert p\Vert_{q}^{q}}{(q-1)\log|\lambda|},\label{eq:L^q dim of ssm}
\end{equation}
\begin{equation}
\dim_{H}\nu_{\lambda,\mathbf{a}}^{p}=\frac{\sum_{i\in\Lambda}p_{i}\log p_{i}}{\log|\lambda|},\label{eq:Haus dim of ssm}
\end{equation}
and
\begin{equation}
\dim_{H}K_{\lambda,\mathbf{a}}=\frac{\log|\Lambda|}{-\log|\lambda|}\:.\label{eq:Haus dim of sss}
\end{equation}

\subsection{Function spaces}

Let $\mathcal{L}$ be the Lebesgue measure on $\mathbb{R}$. For $1\le p\le\infty$
denote by $\Vert\cdot\Vert_{p}$ the $L^{p}$-norm on $L^{p}(\mathbb{R})$.
Given $g\in L^{2}(\mathbb{R})$ write $\widehat{g}$ for the Fourier
transform of $g$, which is again a member of $L^{2}(\mathbb{R})$. 

For $\gamma\in[0,\infty)$ denote by $H_{\gamma}(\mathbb{R})$ the
$2,\gamma$-Sobolev space of $\mathbb{R}$, i.e.
\[
H_{\gamma}(\mathbb{R})=\{g\in L^{2}(\mathbb{R})\::\:\int\:|\widehat{g}(\xi)|^{2}(1+|\xi|^{2})^{\gamma}d\xi<\infty\}\:.
\]
For $g\in H_{\gamma}(\mathbb{R})$ write,
\[
\Vert g\Vert_{(\gamma)}=\left(\int\:|\widehat{g}(\xi)|^{2}(1+|\xi|^{2})^{\gamma}d\xi\right)^{1/2}\:.
\]
Note that $\left(H_{\gamma}(\mathbb{R}),\Vert\cdot\Vert_{(\gamma)}\right)$
is a Hilbert space, $H_{\gamma}(\mathbb{R})\subset H_{\gamma'}(\mathbb{R})$
and $\Vert g\Vert_{(\gamma')}\le\Vert g\Vert_{(\gamma)}$ for each
$0\le\gamma'\le\gamma$, and $H_{0}(\mathbb{R})=L^{2}(\mathbb{R})$.

Given $\mu\in\mathcal{M}(\mathbb{R})$ and a Banach space $B$ of
functions on $\mathbb{R}$, we write $\mu\in B$ when $\mu$ is absolutely
continuous and its density belongs to $B$.

\subsection{\label{subsec:The-exceptional-set}The exceptional set of parameters}

All of our results will be valid for self-similar measures $\nu_{\lambda,\mathbf{a}}^{p}$
and sets $K_{\lambda,\mathbf{a}}$ for which $\lambda\in\mathbb{D}$
lies outside a $0$-dimensional set $\mathcal{E}$. We shall now define
this exceptional set of parameters.

For $X$ which is $\mathbb{R}$ or $\mathbb{C}$ and $\mu\in\mathcal{M}(X)$
write $\widehat{\mu}$ for the Fourier transform of $\mu$, i.e. if
$X=\mathbb{R}$
\[
\widehat{\mu}(\xi)=\int e^{i\xi x}\:d\mu(x)\text{ for }\xi\in\mathbb{R},
\]
and if $X=\mathbb{C}$
\[
\widehat{\mu}(\xi)=\int e^{i\left\langle \xi,w\right\rangle }\:d\mu(w)\text{ for }\xi\in\mathbb{C}\:.
\]

By Theorem D in \cite{SS2} there exists a set $\mathcal{E}'\subset\mathbb{D}$,
with $\dim_{H}\mathcal{E}'=0$, such that the following holds. Let
$\lambda\in\mathbb{D}\setminus(\mathcal{E}'\cup\mathbb{R})$, $\Lambda$
a finite nonempty set, $(a_{i})_{i\in\Lambda}=\mathbf{a}\in\mathbb{C}^{\Lambda}$
with not all of the $a_{i}$ equal, and $p\in\mathbb{P}_{\Lambda}$.
Then there are $C,\gamma>0$ such that
\[
|\widehat{\nu_{\lambda,\mathbf{a}}^{p}}(\xi)|\le C|\xi|^{-\gamma}\text{ for all }\xi\in\mathbb{C}\:.
\]
Write
\[
\mathcal{E}=\cup_{k\ge1}\{\lambda\in\mathbb{D}\::\:\lambda^{k}\in\mathcal{E}'\},
\]
then $\dim_{H}\mathcal{E}=0$.

\section{\label{sec:Statement-of-results}Statement of results}

The following theorem is our main result. Parts (\ref{enu:L^q dens in each dir})
and (\ref{enu:weak contin}) are proven in Section \ref{sec:Proofs-of-parts 1,2},
part (\ref{enu:dc for measues}) is proven in Section \ref{sec:Proof-of-dc},
and part (\ref{enu:equiv of measures}) is proven in Section \ref{sec:Proof-of equiv measures}.
\begin{thm}
\label{thm:main_thm}Let $\lambda\in\mathbb{D}\setminus\mathcal{E}$
be with $\arg\lambda\notin\pi\mathbb{Q}$, $\Lambda$ a finite nonempty
set, $\mathbf{a}\in\mathbb{C}^{\Lambda}$ such that $\mathcal{F}_{\lambda,\mathbf{a}}$
satisfies the SSC, and $p\in\mathbb{P}_{\Lambda}$ with $D_{q}(\nu_{\lambda,\mathbf{a}}^{p})>1$
for some $1<q<\infty$. Set $\nu=\nu_{\lambda,\mathbf{a}}^{p}$, then
\begin{enumerate}
\item \label{enu:L^q dens in each dir}$P_{z}\nu\in L^{q}(\mathbb{R})$
for all $z\in S$;
\item \label{enu:weak contin}$z\rightarrow P_{z}\nu$ is continuous as
a map from $S$ to $L^{q}(\mathbb{R})$ with respect to the weak topology
of $L^{q}(\mathbb{R})$;
\item \label{enu:dc for measues}for every $z\in S$ the measure $\nu_{z,w}$
has exact dimension $\dim_{H}\nu-1$ for $\nu$-a.e. $w\in\mathbb{C}$;
\item \label{enu:equiv of measures}for every $z\in S$ the measures $P_{z}\nu$
and $\mathcal{L}|_{P_{z}(K_{\lambda,\mathbf{a}})}$ are equivalent.
\end{enumerate}
\end{thm}

By differentiation it follows that for every probability vector $p\in\mathbb{P}_{\Lambda}$,
\[
\underset{q\downarrow1}{\lim}\:\frac{\log\Vert p\Vert_{q}^{q}}{(q-1)}=\sum_{i\in\Lambda}p_{i}\log p_{i}\:.
\]
Hence by (\ref{eq:L^q dim of ssm}) and (\ref{eq:Haus dim of ssm})
we get the following direct corollary of Theorem \ref{thm:main_thm}.
\begin{cor}
Let $\lambda\in\mathbb{D}\setminus\mathcal{E}$ be with $\arg\lambda\notin\pi\mathbb{Q}$,
$\Lambda$ a finite nonempty set, $\mathbf{a}\in\mathbb{C}^{\Lambda}$
such that $\mathcal{F}_{\lambda,\mathbf{a}}$ satisfies the SSC, and
$p\in\mathbb{P}_{\Lambda}$ with $\dim_{H}\nu_{\lambda,\mathbf{a}}^{p}>1$.
Then there exists $1<q<\infty$ such that parts (\ref{enu:L^q dens in each dir})
to (\ref{enu:equiv of measures}) in Theorem \ref{thm:main_thm} are
valid for $\nu=\nu_{\lambda,\mathbf{a}}^{p}$.
\end{cor}

\begin{rem}
The following observation might be of interest. Let $\mathcal{F}_{\lambda,\mathbf{a}}$
be an IFS as in Theorem \ref{thm:main_thm} and $p_{1},p_{2}\in\mathbb{P}_{\Lambda}$
with $p_{1}\ne p_{2}$. Set $\nu_{1}=\nu_{\lambda,\mathbf{a}}^{p_{1}}$
and $\nu_{2}=\nu_{\lambda,\mathbf{a}}^{p_{2}}$, and assume
\[
\dim_{H}\nu_{1},\dim_{H}\nu_{2}>1\:.
\]
It holds that $\nu_{1}$ and $\nu_{2}$ are singular. This follows
from the fact that they are both ergodic with respect to an appropriate
map from $K_{\lambda,\mathbf{a}}$ onto itself (see Section \ref{sec:Proof-of-dc})
and since $\nu_{1}\ne\nu_{2}$. On the other hand, by part (\ref{enu:equiv of measures})
of Theorem \ref{thm:main_thm}, $P_{z}\nu_{1}$ is equivalent to $P_{z}\nu_{2}$
for each $z\in S$.
\end{rem}

From Theorem \ref{thm:main_thm} we obtain the following result for
self-similar sets. Its proof is given in Section \ref{sec:Proof-of-Cor for sets}.
For $s\ge0$ denote by $\mathcal{H}^{s}$ and $\mathcal{P}^{s}$ the
$s$-dimensional Hausdorff and packing measures respectively. It is
well known that $0<\mathcal{H}^{s}(K)<\infty$ whenever $K$ is a
self-similar set with the SSC and dimension $s$. Given $z\in S$
and $x\in\mathbb{R}$ we write $K_{z,x}=K\cap P_{z}^{-1}\{x\}$.
\begin{cor}
\label{cor:cor from main thm}Let $\lambda\in\mathbb{D}\setminus\mathcal{E}$
be with $\arg\lambda\notin\pi\mathbb{Q}$, $\Lambda$ a finite nonempty
set, and $\mathbf{a}\in\mathbb{C}^{\Lambda}$ such that $\mathcal{F}_{\lambda,\mathbf{a}}$
satisfies the SSC. Set $K=K_{\lambda,\mathbf{a}}$ and $s=\dim_{H}K$,
and assume $s>1$. Then,
\begin{enumerate}
\item \label{enu:H^s a.s. stat}for each $z\in S$
\[
\dim_{H}K_{z,P_{z}w}=s-1\text{ and }\mathcal{P}^{s-1}(K_{z,P_{z}w})>0\text{ for }\mathcal{H}^{s}\text{-a.e. }w\in K;
\]
\item \label{enu:uni bound proj}there exists $c>0$ such that for every
$z\in S$,
\[
\mathcal{L}\{x\in\mathbb{R}\::\:\dim_{H}K_{z,x}=s-1\text{ and }\mathcal{P}^{s-1}(K_{z,x})>0\}>c\:.
\]
\end{enumerate}
\end{cor}

\begin{rem}
In \cite[Theorem 2.4]{Rap1} the author has shown that, under the
assumptions of the last corollary, for Lebesgue almost every $z\in S$
\[
\mathcal{P}^{s-1}(K_{z,P_{z}w})>0\text{ for }\mathcal{H}^{s}\text{-a.e. }w\in K\:.
\]
In Corollary \ref{cor:cor from main thm} this is established for
all $z\in S$.
\end{rem}

\begin{rem}
If the last corollary would remain true with $\mathcal{H}^{s-1}$
instead of $\mathcal{P}^{s-1}$, then by \cite[Corollary 2.3]{Rap1}
it would follow that $P_{z}(\mathcal{H}^{s}|_{K_{\lambda,\mathbf{a}}})\in L^{\infty}(\mathbb{R})$
for all $z\in S$ and that
\[
\underset{z\in S}{\sup}\:\Vert P_{z}(\mathcal{H}^{s}|_{K_{\lambda,\mathbf{a}}})\Vert_{\infty}<\infty\:.
\]
Unfortunately, currently we are unable to prove this stronger statement.
\end{rem}

If in Theorem \ref{thm:main_thm} we assume that the correlation dimension
$D_{2}(\nu_{\lambda,\mathbf{a}}^{p})$ exceeds $1$ then we get the
following result. It guarantees the continuity of the map $z\rightarrow P_{z}\nu$
with respect to the norm topology, which is of course much stronger
than continuity with respect to the weak topology. The proof is given
in Section \ref{sec:Proofs-of-parts 1,2}.
\begin{thm}
\label{thm:continuous frac deriv}Let $\lambda\in\mathbb{D}\setminus\mathcal{E}$
be with $\arg\lambda\notin\pi\mathbb{Q}$, $\Lambda$ a finite nonempty
set, $\mathbf{a}\in\mathbb{C}^{\Lambda}$ such that $\mathcal{F}_{\lambda,\mathbf{a}}$
satisfies the SSC, and $p\in\mathbb{P}_{\Lambda}$ with $D_{2}(\nu_{\lambda,\mathbf{a}}^{p})>1$.
Set $\nu=\nu_{\lambda,\mathbf{a}}^{p}$, then there exists $\gamma>0$
such that
\begin{enumerate}
\item \label{enu:frac deriv in each dir}$P_{z}\nu\in H_{\gamma}(\mathbb{R})$
for all $z\in S$;
\item \label{enu:cont of proj}$z\rightarrow P_{z}\nu$ is continuous as
a map from $S$ to $\left(H_{\gamma}(\mathbb{R}),\Vert\cdot\Vert_{(\gamma)}\right)$.
\end{enumerate}
\end{thm}

\begin{rem}
Note that part (\ref{enu:cont of proj}) of the last theorem implies
that $z\rightarrow P_{z}\nu$ is continuous as a map from $S$ to
$\left(L^{2}(\mathbb{R}),\Vert\cdot\Vert_{2}\right)$. This in turn
implies the continuity of this map with respect to the total variation
norm on $\mathcal{M}(\mathbb{R})$.
\end{rem}

Let $\lambda\in\mathbb{D}$ and $\mathbf{a}\in\mathbb{C}^{\Lambda}$
be such that $\mathcal{F}_{\lambda,\mathbf{a}}$ satisfies the SSC,
and write $s=\dim_{H}K_{\lambda,\mathbf{a}}$ and $p'=(\frac{1}{|\Lambda|},...,\frac{1}{|\Lambda|})$.
It is not hard to see that there exists $0<c<\infty$ so that
\[
\mathcal{H}^{s}|_{K_{\lambda,\mathbf{a}}}=c\cdot\nu_{\lambda,\mathbf{a}}^{p'}\:.
\]
Also, for each $1<q<\infty$ it follows by (\ref{eq:L^q dim of ssm})
and (\ref{eq:Haus dim of sss}) that
\[
D_{q}(\nu_{\lambda,\mathbf{a}}^{p'})=\frac{\log|\Lambda|}{-\log|\lambda|}=\dim_{H}K_{\lambda,\mathbf{a}}\:.
\]
From these facts we get the following direct corollary of Theorems
\ref{thm:main_thm} and \ref{thm:continuous frac deriv}. It shows
that when $s>1$ the projections of $\mathcal{H}^{s}|_{K_{\lambda,\mathbf{a}}}$
are very regular.
\begin{cor}
Let $\lambda\in\mathbb{D}\setminus\mathcal{E}$ be with $\arg\lambda\notin\pi\mathbb{Q}$,
$\Lambda$ a finite nonempty set, and $\mathbf{a}\in\mathbb{C}^{\Lambda}$
such that $\mathcal{F}_{\lambda,\mathbf{a}}$ satisfies the SSC. Set
$s=\dim_{H}K_{\lambda,\mathbf{a}}$ and assume $s>1$. Then for $\nu=\mathcal{H}^{s}|_{K_{\lambda,\mathbf{a}}}$
and every $1<q<\infty$, parts (\ref{enu:L^q dens in each dir}) to
(\ref{enu:equiv of measures}) of Theorem \ref{thm:main_thm} are
valid. Additionally, parts (\ref{enu:frac deriv in each dir}) and
(\ref{enu:cont of proj}) of Theorem \ref{thm:continuous frac deriv}
are valid for some $\gamma>0$.
\end{cor}

\section{\label{sec:Proofs-of-parts 1,2}Proof of parts (\ref{enu:L^q dens in each dir})
and (\ref{enu:weak contin}) of Theorems \ref{thm:main_thm} and \ref{thm:continuous frac deriv}}

The following two results will be needed. Let $X$ be $\mathbb{C}$
or $\mathbb{R}$ and for $\mu,\nu\in\mathcal{M}(X)$ write $\mu*\nu$
for their convolution.
\begin{thm}
[\cite{SS1}{, Theorem 4.4}]\label{thm:shmerkin and solomyak}Let
$1<q<\infty$, $\mu\in\mathcal{M}(\mathbb{R})$ with $D_{q}(\mu)=1$,
and $\nu\in\mathcal{M}(\mathbb{R})$ such that there exist $C,\gamma>0$
with $|\widehat{\nu}(\xi)|\le C|\xi|^{-\gamma}$ for all $\xi\in\mathbb{R}$.
Then $\mu*\nu\in L^{q}(\mathbb{R})$.
\end{thm}

\begin{lem}
[\cite{Sh1}{, Lemma 2.1}]\label{lem:shmerkin}Let $\mu\in\mathcal{M}(\mathbb{R})$
with $D_{2}(\mu)=1$ and $\nu\in\mathcal{M}(\mathbb{R})$ such that
there exist $C,\gamma>0$ with $|\widehat{\nu}(\xi)|\le C|\xi|^{-\gamma}$
for all $\xi\in\mathbb{R}$. Then $\mu*\nu\in H_{\gamma/4}(\mathbb{R})$.
\end{lem}

In the next proposition we prove the first part of Theorems \ref{thm:main_thm}
and \ref{thm:continuous frac deriv}. The proof follows an idea introduced
in \cite{Sh1}. It also relies on a result from \cite{Sh2} regarding
the $L^{q}$ dimensions of projections of planar self-similar measures.
\begin{prop}
\label{prop:ac in each dir}Let $\lambda\in\mathbb{D}\setminus\mathcal{E}$
be with $\arg\lambda\notin\pi\mathbb{Q}$, $\Lambda$ a finite nonempty
set, $(a_{i})_{i\in\Lambda}=\mathbf{a}\in\mathbb{C}^{\Lambda}$ such
that $\mathcal{F}_{\lambda,\mathbf{a}}$ satisfies the SSC, and $p\in\mathbb{P}_{\Lambda}$.
Set $\nu=\nu_{\lambda,\mathbf{a}}^{p}$, then
\begin{enumerate}
\item \label{enu:L^q in each dir}if $1<q<\infty$ is such that $D_{q}(\nu)>1$
then $P_{z}\nu\in L^{q}(\mathbb{R})$ for all $z\in S$;
\item \label{enu:H_gamma in each dir}if $D_{2}(\nu)>1$ then there exists
$\gamma>0$ such that $P_{z}\nu\in H_{\gamma}(\mathbb{R})$ for all
$z\in S$.
\end{enumerate}
\end{prop}

\begin{proof}
Without loss of generality we can assume that $a_{i}=0$ for some
$i\in\Lambda$. Otherwise we can arrange this by replacing $\mathcal{F}_{\lambda,\mathbf{a}}=\{f_{\lambda,\mathbf{a}}^{i}\}_{i\in\Lambda}$
with
\[
\{h\circ f_{\lambda,\mathbf{a}}^{i}\circ h^{-1}\}_{i\in\Lambda}
\]
and $\nu$ with $h\nu$, where $h:\mathbb{C}\rightarrow\mathbb{C}$
is of the form $h(w)=w+\beta$ for some $\beta\in\mathbb{C}$.

Assume that $1<q<\infty$ satisfies $D_{q}(\nu)>1$. Let $k\ge2$
be with $(1-\frac{1}{k})D_{q}(\nu)>1$. For $(i_{1},...,i_{k-1})=\boldsymbol{i}\in\Lambda^{k-1}$
write $p_{\boldsymbol{i}}=p_{i_{1}}\cdot...\cdot p_{i_{k-1}}$ and
\[
g_{\boldsymbol{i}}(w)=\lambda^{k}w+\sum_{j=1}^{k-1}a_{i_{j}}\lambda^{j}\text{ for }w\in\mathbb{C}\:.
\]
Since $\mathcal{F}_{\lambda,\mathbf{a}}$ satisfies the SSC and $a_{i}=0$
for some $i\in\Lambda$, it is easy to see that $\{g_{\boldsymbol{i}}\}_{\boldsymbol{i}\in\Lambda^{k-1}}$
also satisfies the SSC. Let $\mu\in\mathcal{M}(\mathbb{C})$ be the
self-similar measure corresponding to $\{g_{\boldsymbol{i}}\}_{\boldsymbol{i}\in\Lambda^{k-1}}$
and $\{p_{\boldsymbol{i}}\}_{\boldsymbol{i}\in\Lambda^{k-1}}$, i.e.
\[
\mu=\sum_{\boldsymbol{i}\in\Lambda^{k-1}}p_{\boldsymbol{i}}\cdot g_{\boldsymbol{i}}\mu\:.
\]
By (\ref{eq:L^q dim of ssm}),
\begin{eqnarray*}
D_{q}(\mu) & = & \frac{\log\Vert(p_{\boldsymbol{i}})_{\boldsymbol{i}\in\Lambda^{k-1}}\Vert_{q}^{q}}{(q-1)\log|\lambda^{k}|}\\
 & = & \frac{(k-1)\log\Vert p\Vert_{q}^{q}}{k(q-1)\log|\lambda|}=(1-\frac{1}{k})D_{q}(\nu)>1\:.
\end{eqnarray*}
Hence by Theorem 8.2 in \cite{Sh2},
\begin{equation}
D_{q}(P_{z}\mu)=1\text{ for all }z\in S\:.\label{eq:D_q of proj}
\end{equation}

Recall the definition of the sets $\mathcal{E}$ and $\mathcal{E}'$
from Section \ref{subsec:The-exceptional-set}. Since $\lambda\notin\mathcal{E}$
and $\arg\lambda\notin\pi\mathbb{Q}$ we have $\lambda^{k}\notin\mathcal{E}'\cup\mathbb{R}$.
Hence there exist $C,\gamma>0$ such that,
\[
|\widehat{\nu_{\lambda^{k},\mathbf{a}}^{p}}(\xi)|\le C|\xi|^{-\gamma}\text{ for all }\xi\in\mathbb{C}\:.
\]
Let $z\in S$, then a direct computation shows that
\[
\widehat{P_{z}\nu_{\lambda^{k},\mathbf{a}}^{p}}(t)=\widehat{\nu_{\lambda^{k},\mathbf{a}}^{p}}(tz)\text{ for all }t\in\mathbb{R},
\]
and so
\begin{equation}
|\widehat{P_{z}\nu_{\lambda^{k},\mathbf{a}}^{p}}(t)|\le C|t|^{-\gamma}\text{ for all }t\in\mathbb{R}\:.\label{eq:four decay nu}
\end{equation}
Note that $\nu=\mu*\nu_{\lambda^{k},\mathbf{a}}^{p}$, hence $P_{z}\nu=P_{z}\mu*P_{z}\nu_{\lambda^{k},\mathbf{a}}^{p}$.
From this, (\ref{eq:D_q of proj}), (\ref{eq:four decay nu}), and
Theorem \ref{thm:shmerkin and solomyak}, it now follows $P_{z}\nu\in L^{q}(\mathbb{R})$,
which completes the proof of the first part.

The proof of the second part is similar, except that at the end of
the proof one needs to use Lemma \ref{lem:shmerkin} instead of Theorem
\ref{thm:shmerkin and solomyak}.
\end{proof}
We now turn to the proof of the second part of Theorems \ref{thm:main_thm}
and \ref{thm:continuous frac deriv}. Throughout this section the
pair $(B,\Vert\cdot\Vert_{B})$ will denote $(L^{q}(\mathbb{R}),\Vert\cdot\Vert_{q})$
for some $1<q<\infty$ or $(H_{\gamma}(\mathbb{R}),\Vert\cdot\Vert_{(\gamma)})$
for some $0\le\gamma<\infty$. We write $C_{c}(\mathbb{R})$ for the
collection of all compactly supported continuous functions on $\mathbb{R}$.
\begin{lem}
\label{lem:closed subset of S}Let $\nu\in\mathcal{M}(\mathbb{C})$
and $C>0$ be given, and set
\[
F=\{z\in S\::\:P_{z}\nu\in B\text{ and }\Vert\frac{dP_{z}\nu}{d\mathcal{L}}\Vert_{B}\le C\}\:.
\]
Then $F$ is a closed subset of $S$.
\end{lem}

\begin{proof}
For every $z\in F$ there exists $g_{z}\in B$ with $dP_{z}\nu=g_{z}d\mathcal{L}$
and $\Vert g_{z}\Vert_{B}\le C$. Let $z_{1},z_{2},...\in F$ and
$z\in S$ be with $z_{k}\overset{k}{\rightarrow}z$ as $k\rightarrow\infty$.
Note that $\{g_{z_{k}}\}_{k\ge1}$ is a bounded sequence in the reflexive
Banach space $B$. Hence, by moving to a subsequence without changing
notation, we may assume that there exists $g\in B$ such that
\begin{equation}
g_{z_{k}}\overset{k}{\rightarrow}g\text{ weakly in \ensuremath{B} as \ensuremath{k\rightarrow\infty}\:.}\label{eq:weak convergence in B}
\end{equation}
Additionally, by (\ref{eq:weak convergence in B})
\[
\Vert g\Vert_{B}\le\underset{k}{\liminf}\:\Vert g_{z_{k}}\Vert_{B}\le C\:.
\]

Given $h\in C_{c}(\mathbb{R})$ we get from (\ref{eq:weak convergence in B}),
\[
\int h(x)g(x)\:dx=\underset{k}{\lim}\:\int h(x)g_{z_{k}}(x)\:dx=\underset{k}{\lim}\:\int h(P_{z_{k}}w)\:d\nu(w)\:.
\]
Hence, by bounded convergence and since $P_{z_{k}}w\overset{k}{\rightarrow}P_{z}w$
for all $w\in\mathbb{C}$,
\[
\int h(x)g(x)\:dx=\int h(P_{z}w)\:d\nu(w)=\int h(x)\:dP_{z}\nu(x)\:.
\]
This shows $dP_{z}\nu=gd\mathcal{L}$, which implies that $z\in F$
and completes the proof of the lemma.
\end{proof}
For $b\in\mathbb{R}$ and $c>0$ let
\[
\tau_{b}(x)=x-b\text{ and }M_{c}(x)=cx\text{ for }x\in\mathbb{R}\:.
\]
In what follows $\Lambda$ stands for some finite nonempty index set.
Recall the definition of the IFS $\mathcal{F}_{\lambda,\mathbf{a}}$
from (\ref{eq:def of IFS}), which consists of maps $f_{i}=f_{\lambda,\mathbf{a}}^{i}$
for $i\in\Lambda$. Given a word $i_{1}...i_{k}=\boldsymbol{i}\in\Lambda^{*}$
write $f_{\boldsymbol{i}}=f_{i_{1}}\circ...\circ f_{i_{k}}$, and
for $(p_{i})_{i\in\Lambda}=p\in\mathbb{P}_{\Lambda}$ set $p_{\boldsymbol{i}}=p_{i_{1}}\cdot...\cdot p_{i_{k}}$.
\begin{lem}
\label{lem:SS of densities}Let $\alpha\in S$, $0<r<1$, $\mathbf{a}\in\mathbb{C}^{\Lambda}$,
and $p\in\mathbb{P}_{\Lambda}$. Write $\lambda=r\alpha$, $\nu=\nu_{\lambda,\mathbf{a}}^{p}$,
and $f_{i}=f_{\lambda,\mathbf{a}}^{i}$ for every $i\in\Lambda$.
Assume that $P_{z}\nu$ is absolutely continuous for each $z\in S$
and write $g_{z}=\frac{dP_{z}\nu}{d\mathcal{L}}$. Let $z\in S$ and
$k\ge1$, and for each $\boldsymbol{i}\in\Lambda^{k}$ set $b_{\boldsymbol{i}}=P_{z}f_{\boldsymbol{i}}(0)$.
Then for $\mathcal{L}$-a.e. $x\in\mathbb{R}$,
\[
g_{z}(x)=\sum_{\boldsymbol{i}\in\Lambda^{k}}p_{\boldsymbol{i}}r^{-k}\cdot g_{\alpha^{-k}z}\left(\tau_{r^{-k}b_{\boldsymbol{i}}}M_{r^{-k}}x\right)\:.
\]
\end{lem}

\begin{proof}
For $h\in C_{c}(\mathbb{R})$,
\[
\int h(x)g_{z}(x)\:dx=\int h(P_{z}w)\:d\nu(w)=\sum_{\boldsymbol{i}\in\Lambda^{k}}p_{\boldsymbol{i}}\int h(P_{z}w)\:df_{\boldsymbol{i}}\nu(w)\:.
\]
For $\boldsymbol{i}\in\Lambda^{k}$,
\begin{multline*}
\int h(P_{z}f_{\boldsymbol{i}}w)\:d\nu(w)=\int h(b_{\boldsymbol{i}}+r^{k}P_{\alpha^{-k}z}w)\:d\nu(w)\\
=\int h(b_{\boldsymbol{i}}+r^{k}x)\:dP_{\alpha^{-k}z}\nu(x)=\int h(b_{\boldsymbol{i}}+r^{k}x)g_{\alpha^{-k}z}(x)\:dx\\
=\int h(x)g_{\alpha^{-k}z}(r^{-k}x-r^{-k}b_{\boldsymbol{i}})r^{-k}\:dx\:.
\end{multline*}
By combining these equalities,
\[
\int h(x)g_{z}(x)\:dx=\int h(x)\sum_{\boldsymbol{i}\in\Lambda^{k}}p_{\boldsymbol{i}}r^{-k}\cdot g_{\alpha^{-k}z}(\tau_{r^{-k}b_{\boldsymbol{i}}}M_{r^{-k}}x)\:dx\:.
\]
Now since this holds for every $h\in C_{c}(\mathbb{R})$ the lemma
follows.
\end{proof}
\begin{lem}
\label{lem:bounded opp norm of mult}There exists $\beta\in\mathbb{R}$
such that for every $g\in B$, $b\in\mathbb{R}$ and $c\ge1$,
\[
\Vert g\circ\tau_{b}\circ M_{c}\Vert_{B}\le c^{\beta}\cdot\Vert g\Vert_{B}\:.
\]
\end{lem}

\begin{proof}
Let $g\in B$, $b\in\mathbb{R}$ and $c\ge1$. If $B=L^{q}(\mathbb{R})$
for $1<q<\infty$, then
\[
\Vert g\circ\tau_{b}\circ M_{c}\Vert_{B}=c^{-1/q}\Vert g\Vert_{B}\:.
\]
If $B=H_{\gamma}(\mathbb{R})$ for $0\le\gamma<\infty$, then
\begin{eqnarray*}
\Vert g\circ\tau_{b}\circ M_{c}\Vert_{B}^{2} & = & \int|(g\circ\tau_{b}\circ M_{c})^{\wedge}(\xi)|^{2}(1+|\xi|^{2})^{\gamma}\:d\xi\\
 & = & \int|c^{-1}\cdot e^{ib\xi/c}\cdot\widehat{g}(\xi/c)|^{2}(1+|\xi|^{2})^{\gamma}\:d\xi\\
 & = & c^{-1}\int|\widehat{g}(\xi)|^{2}(1+|c\xi|^{2})^{\gamma}\:d\xi\le c^{2\gamma-1}\Vert g\Vert_{B}^{2},
\end{eqnarray*}
which proves the lemma.
\end{proof}
The following key proposition will be used several times below. Recall
that $(B,\Vert\cdot\Vert_{B})$ denotes $(L^{q}(\mathbb{R}),\Vert\cdot\Vert_{q})$
for some $1<q<\infty$ or $(H_{\gamma}(\mathbb{R}),\Vert\cdot\Vert_{(\gamma)})$
for some $0\le\gamma<\infty$.
\begin{prop}
\label{prop:bounded norms}Let $\lambda\in\mathbb{D}$ be with $\arg\lambda\notin\pi\mathbb{Q}$,
$\mathbf{a}\in\mathbb{C}^{\Lambda}$, and $p\in\mathbb{P}_{\Lambda}$.
Set $\nu=\nu_{\lambda,\mathbf{a}}^{p}$ and assume that $P_{z}\nu\in B$
for all $z\in S$. Then $\{P_{z}\nu\}_{z\in S}$ is a bounded subset
of $(B,\Vert\cdot\Vert_{B})$.
\end{prop}

\begin{proof}
Set $f_{i}=f_{\lambda,\mathbf{a}}^{i}$ for every $i\in\Lambda$.
By assumption, for every $z\in S$ there exists $g_{z}\in B$ with
$dP_{z}\nu=g_{z}d\mathcal{L}$. For $n\ge1$ let
\[
F_{n}=\{z\in S\::\:\Vert g_{z}\Vert_{B}\le n\},
\]
then $F_{n}$ is closed in $S$ by Lemma \ref{lem:closed subset of S}.

From $S=\cup_{n\ge1}F_{n}$ and Baire's theorem, it follows that there
exist $n\ge1$ and an open nonempty subset $V$ of $S$ with $V\subset F_{n}$.
Let $\alpha\in S$ and $0<r<1$ be with $\lambda=r\alpha$. Since
\[
\arg\alpha=\arg\lambda\notin\pi\mathbb{Q},
\]
 there exits $N\ge1$ such that for each $z\in S$ there exists $1\le k\le N$
with $\alpha^{-k}z\in V$. Fix $z\in S$ and let $1\le k\le N$ be
with $\alpha^{-k}z\in V$.

For $\boldsymbol{i}\in\Lambda^{k}$ write $b_{\boldsymbol{i}}=P_{z}f_{\boldsymbol{i}}(0)$.
By Lemma \ref{lem:SS of densities} it follows that for $\mathcal{L}$-a.e.
$x\in\mathbb{R}$,
\[
g_{z}(x)=\sum_{\boldsymbol{i}\in\Lambda^{k}}p_{\boldsymbol{i}}r^{-k}\cdot g_{\alpha^{-k}z}\left(\tau_{r^{-k}b_{\boldsymbol{i}}}M_{r^{-k}}x\right)\:.
\]
Hence by Lemma \ref{lem:bounded opp norm of mult} and since $\alpha^{-k}z\in V\subset F_{n}$,
\begin{eqnarray*}
\Vert g_{z}\Vert_{B} & \le & \sum_{\boldsymbol{i}\in\Lambda^{k}}p_{\boldsymbol{i}}r^{-k}\cdot\Vert g_{\alpha^{-k}z}\circ\tau_{r^{-k}b_{\boldsymbol{i}}}\circ M_{r^{-k}}\Vert_{B}\\
 & \le & r^{-k(1+\beta)}\cdot\Vert g_{\alpha^{-k}z}\Vert_{B}\le n\cdot\max\{r^{-N(1+\beta)},1\},
\end{eqnarray*}
which completes the proof of the proposition.
\end{proof}
Part (\ref{enu:weak contin}) of Theorems \ref{thm:main_thm} and
\ref{thm:continuous frac deriv} will follow directly from the following
two lemmas and Propositions \ref{prop:ac in each dir} and \ref{prop:bounded norms}.
\begin{lem}
\label{lem:weak continuity}Let $\nu\in\mathcal{M}(\mathbb{C})$,
$1<q<\infty$ and $C>0$. Assume that $P_{z}\nu\in L^{q}(\mathbb{R})$,
with $\Vert P_{z}\nu\Vert_{q}\le C$, for each $z\in S$. Then $z\rightarrow P_{z}\nu$
is continuous as a map from $S$ to $L^{q}(\mathbb{R})$, with respect
to the weak topology of $L^{q}(\mathbb{R})$.
\end{lem}

\begin{proof}
By assumption for every $z\in S$ there exists $g_{z}\in L^{q}(\mathbb{R})$
with $dP_{z}\nu=g_{z}d\mathcal{L}$. Moreover,
\begin{equation}
\Vert g_{z}\Vert_{q}\le C\text{ for all }z\in S\:.\label{eq:uniform norm bound}
\end{equation}
Let $z\in S$ and $\{z_{0,k}\}_{k\ge1}\subset S$ be with $z_{0,k}\overset{k}{\rightarrow}z$.
It suffices to show that
\begin{equation}
g_{z_{0,k}}\overset{k}{\rightarrow}g_{z}\text{ weakly in }L^{q}(\mathbb{R})\text{ as }k\rightarrow\infty\:.\label{eq:weak convergence}
\end{equation}

Let $\{z_{1,k}\}_{k\ge1}$ be a subsequence of $\{z_{0,k}\}_{k\ge1}$.
From (\ref{eq:uniform norm bound}) and since $L^{q}(\mathbb{R})$
is reflexive, it follows that there exist a subsequence $\{z_{2,k}\}_{k\ge1}$
of $\{z_{1,k}\}_{k\ge1}$ and $g\in L^{q}(\mathbb{R})$ such that
$g_{z_{2,k}}\overset{k}{\rightarrow}g$ weakly in $L^{q}(\mathbb{R})$
as $k\rightarrow\infty$.

Since $z_{2,k}\overset{k}{\rightarrow}z$, it holds for each $h\in C_{c}(\mathbb{R})$
that
\begin{multline*}
\int h(x)g(x)\:dx=\underset{k}{\lim}\:\int h(x)g_{z_{2,k}}(x)\:dx\\
=\underset{k}{\lim}\:\int h(P_{z_{2,k}}w)\:d\nu(w)=\int h(P_{z}w)\:d\nu(w)=\int h(x)g_{z}(x)\:dx\:.
\end{multline*}
Hence $g=g_{z}$, and so $g_{z_{2,k}}\overset{k}{\rightarrow}g_{z}$
weakly as $k\rightarrow\infty$. Since $\{z_{1,k}\}_{k\ge1}$ was
an arbitrary subsequence of $\{z_{0,k}\}_{k\ge1}$ this gives (\ref{eq:weak convergence})
and completes the proof of the lemma.
\end{proof}
\begin{lem}
\label{lem:norm continuity}Let $\nu\in\mathcal{M}(\mathbb{C})$ and
$\gamma,C>0$. Assume that $P_{z}\nu\in H_{\gamma}(\mathbb{R})$,
with $\Vert P_{z}\nu\Vert_{(\gamma)}\le C$, for all $z\in S$. Let
$\gamma'\in(0,\gamma)$, then $z\rightarrow P_{z}\nu$ is continuous
as a map from $S$ to $\left(H_{\gamma'}(\mathbb{R}),\Vert\cdot\Vert_{(\gamma')}\right)$.
\end{lem}

\begin{proof}
By assumption for every $z\in S$ there exists $g_{z}\in H_{\gamma}(\mathbb{R})$
with $dP_{z}\nu=g_{z}d\mathcal{L}$. Moreover,
\begin{equation}
\Vert g_{z}\Vert_{(\gamma)}\le C\text{ for all }z\in S\:.\label{eq:uniform frac norm bound}
\end{equation}
Let $z\in S$ and $\{z_{0,k}\}_{k\ge1}\subset S$ be with $z_{0,k}\overset{k}{\rightarrow}z$.
It suffices to show that
\begin{equation}
\Vert g_{z}-g_{z_{0,k}}\Vert_{(\gamma')}\overset{k}{\rightarrow}0\text{ as }k\rightarrow\infty\:.\label{eq:convergence in norm}
\end{equation}

Let $\{z_{1,k}\}_{k\ge1}$ be a subsequence of $\{z_{0,k}\}_{k\ge1}$.
There exists a bounded interval $J\subset\mathbb{R}$ with
\[
supp(P_{z}\nu)\subset J\text{ for all }z\in S\:.
\]
From this, from (\ref{eq:uniform frac norm bound}), and by Rellich's
Theorem (see for instance \cite[Theorem 9.22]{F2}), there exist a
subsequence $\{z_{2,k}\}_{k\ge1}$ of $\{z_{1,k}\}_{k\ge1}$ and $g\in H_{\gamma'}(\mathbb{R})$
such that $\Vert g-g_{z_{2,k}}\Vert_{(\gamma')}\overset{k}{\rightarrow}0$.

As in the proof of the previous lemma, it can be shown that $g=g_{z}$
and so $\Vert g_{z}-g_{z_{2,k}}\Vert_{(\gamma')}\overset{k}{\rightarrow}0$.
Since $\{z_{1,k}\}_{k\ge1}$ was an arbitrary subsequence of $\{z_{0,k}\}_{k\ge1}$
this gives (\ref{eq:convergence in norm}) and completes the proof
of the lemma.
\end{proof}
\begin{proof}[Proof of part (\ref{enu:weak contin}) of Theorems \ref{thm:main_thm}
and \ref{thm:continuous frac deriv}]
Let $\lambda\in\mathbb{D}\setminus\mathcal{E}$ be with $\arg\lambda\notin\pi\mathbb{Q}$,
$\mathbf{a}\in\mathbb{C}^{\Lambda}$ such that $\mathcal{F}_{\lambda,\mathbf{a}}$
satisfies the SSC, and $p\in\mathbb{P}_{\Lambda}$ with $D_{q}(\nu_{\lambda,\mathbf{a}}^{p})>1$
for some $1<q<\infty$. Set $\nu=\nu_{\lambda,\mathbf{a}}^{p}$, then
by part (\ref{enu:L^q in each dir}) of Proposition \ref{prop:ac in each dir}
it follows that $P_{z}\nu\in L^{q}(\mathbb{R})$ for all $z\in S$.
From this and Proposition \ref{prop:bounded norms} we get that $\{P_{z}\nu\}_{z\in S}$
is a bounded subset of $(L^{q}(\mathbb{R}),\Vert\cdot\Vert_{q})$.
By Lemma \ref{lem:weak continuity} it now follows that $z\rightarrow P_{z}\nu$
is continuous as a map from $S$ to $L^{q}(\mathbb{R})$, with respect
to the weak topology of $L^{q}(\mathbb{R})$, which completes the
proof of part (\ref{enu:weak contin}) of Theorem \ref{thm:main_thm}.

Part (\ref{enu:cont of proj}) of Theorem \ref{thm:continuous frac deriv}
follows in a similar manner, except that one needs to use part (\ref{enu:H_gamma in each dir})
of Proposition \ref{prop:ac in each dir} and Lemma \ref{lem:norm continuity}.
\end{proof}

\section{\label{sec:Proof-of-dc}Proof of part (\ref{enu:dc for measues})
of Theorem \ref{thm:main_thm}}

Throughout this section fix $\alpha\in S$, $0<r<1$, $\Lambda$ a
finite nonempty index set, $\mathbf{a}\in\mathbb{C}^{\Lambda}$, and
$p\in\mathbb{P}_{\Lambda}$. Write $\lambda=r\alpha$, $K=K_{\lambda,\mathbf{a}}$,
$\nu=\nu_{\lambda,\mathbf{a}}^{p}$, and $f_{i}=f_{\lambda,\mathbf{a}}^{i}$
for every $i\in\Lambda$. Assume that $\mathcal{F}_{\lambda,\mathbf{a}}$
satisfies the SSC. Assume further that $P_{z}\nu$ is absolutely continuous
for each $z\in S$ and write $g_{z}=\frac{dP_{z}\nu}{d\mathcal{L}}$.

Given $i_{1}...i_{k}=\boldsymbol{i}\in\Lambda^{*}$ recall that
\[
f_{\boldsymbol{i}}=f_{i_{1}}\circ...\circ f_{i_{k}}\text{ and }p_{\boldsymbol{i}}=p_{i_{1}}\cdot...\cdot p_{i_{k}},
\]
and set $K_{\boldsymbol{i}}=f_{\boldsymbol{i}}(K)$. For $w\in K$
and $k\ge1$ denote by $\boldsymbol{i}_{k}(w)$ the unique word of
length $k$ over $\Lambda$ with $w\in K_{\boldsymbol{i}_{k}(w)}$.
Let $T:K\rightarrow K$ be with $Tw=f_{\boldsymbol{i}_{1}(w)}^{-1}(w)$
for $w\in K$. It is easy to verify that the system $(K,T,\nu)$ is
measure preserving and ergodic.
\begin{lem}
\label{lem:slices by densities}Let $z\in S$ and $k\ge1$, then for
$\nu$-a.e. $w\in K$
\[
\nu_{z,w}(K_{\boldsymbol{i}_{k}(w)})=\frac{g_{\alpha^{-k}z}(P_{\alpha^{-k}z}T^{k}w)}{g_{z}(P_{z}w)}\cdot p_{\boldsymbol{i}_{k}(w)}r^{-k}\:.
\]
\end{lem}

\begin{proof}
From Lemma 3.3 in \cite{FH} we get that for each $\eta\in S$ and
Borel set $A\subset\mathbb{C}$,
\begin{equation}
\nu_{\eta,w}(A)=\underset{\delta\downarrow0}{\lim}\:\frac{\nu(P_{\eta}^{-1}(B(P_{\eta}w,\delta))\cap A)}{\nu(P_{\eta}^{-1}(B(P_{\eta}w,\delta)))}\text{ for \ensuremath{\nu}-a.e. \ensuremath{w\in\mathbb{C}}\:.}\label{eq:slices as limit}
\end{equation}
By Theorem 2.12 in \cite{M} it follows that for each $\eta\in S$,
\begin{equation}
g_{\eta}(P_{\eta}w)=\underset{\delta\downarrow0}{\lim}\:\frac{P_{\eta}\nu(B(P_{\eta}w,\delta))}{2\delta}\text{ for \ensuremath{\nu}-a.e. \ensuremath{w\in\mathbb{C}}\:.}\label{eq:densities as limit}
\end{equation}
From (\ref{eq:slices as limit}) and (\ref{eq:densities as limit})
we get that for $\nu$-a.e. $w\in K$,
\begin{eqnarray}
\nu_{z,w}(K_{\boldsymbol{i}_{k}(w)}) & = & \underset{\delta\downarrow0}{\lim}\:\frac{\nu(K_{\boldsymbol{i}_{k}(w)}\cap P_{z}^{-1}(B(P_{z}w,\delta)))}{\nu(P_{z}^{-1}(B(P_{z}w,\delta)))}\nonumber \\
 & = & \underset{\delta\downarrow0}{\lim}\:\frac{2\delta}{\nu(P_{z}^{-1}(B(P_{z}w,\delta)))}\cdot\frac{\nu(K_{\boldsymbol{i}_{k}(w)}\cap P_{z}^{-1}(B(P_{z}w,\delta)))}{2\delta}\nonumber \\
 & = & g_{z}(P_{z}w)^{-1}\underset{\delta\downarrow0}{\lim}\:\frac{\nu(K_{\boldsymbol{i}_{k}(w)}\cap P_{z}^{-1}(B(P_{z}w,\delta)))}{2\delta}\:.\label{eq:stage 1 of proof}
\end{eqnarray}

For each $\delta>0$ set
\[
F_{\delta}=P_{\alpha^{-k}z}^{-1}(B(P_{\alpha^{-k}z}f_{\boldsymbol{i}_{k}(w)}^{-1}(w),\delta r^{-k}))
\]
and for $\eta\in S$ write $\eta^{\perp}=e^{i\pi/2}\eta$. We have,
\begin{eqnarray*}
P_{z}^{-1}(B(P_{z}w,\delta)) & = & w+z^{\perp}\mathbb{R}+B(0,\delta)\\
 & = & f_{\boldsymbol{i}_{k}(w)}\circ f_{\boldsymbol{i}_{k}(w)}^{-1}(w+z^{\perp}\mathbb{R}+B(0,\delta))\\
 & = & f_{\boldsymbol{i}_{k}(w)}(f_{\boldsymbol{i}_{k}(w)}^{-1}(w)+(\alpha^{-k}z)^{\perp}\mathbb{R}+B(0,\delta r^{-k}))\\
 & = & f_{\boldsymbol{i}_{k}(w)}(F_{\delta})\:.
\end{eqnarray*}
From this, from (\ref{eq:stage 1 of proof}) and since $\nu$ is self-similar,
\begin{eqnarray}
\nu_{z,w}(K_{\boldsymbol{i}_{k}(w)}) & = & g_{z}(P_{z}w)^{-1}\underset{\delta\downarrow0}{\lim}\:\frac{\nu(K_{\boldsymbol{i}_{k}(w)}\cap f_{\boldsymbol{i}_{k}(w)}(F_{\delta}))}{2\delta}\nonumber \\
 & = & g_{z}(P_{z}w)^{-1}\underset{\delta\downarrow0}{\lim}\:\frac{1}{2\delta}\sum_{\boldsymbol{i}\in\Lambda^{k}}p_{\boldsymbol{i}}\cdot\nu(f_{\boldsymbol{i}}^{-1}(f_{\boldsymbol{i}_{k}(w)}(K\cap F_{\delta})))\:.\label{eq:result of ss of nu}
\end{eqnarray}

Since $\mathcal{F}_{\lambda,\mathbf{a}}$ satisfies the SSC,
\[
K\cap f_{\boldsymbol{i}}^{-1}(f_{\boldsymbol{i}_{k}(w)}(K))=\emptyset\text{ for each }\boldsymbol{i}\in\Lambda^{k}\setminus\{\boldsymbol{i}_{k}(w)\}\:.
\]
Hence by (\ref{eq:result of ss of nu}) and (\ref{eq:densities as limit}),
\begin{eqnarray*}
\nu_{z,w}(K_{\boldsymbol{i}_{k}(w)}) & = & g_{z}(P_{z}w)^{-1}\underset{\delta\downarrow0}{\lim}\:p_{\boldsymbol{i}_{k}(w)}r^{-k}\cdot\frac{\nu(K\cap F_{\delta})}{2\delta r^{-k}}\\
 & = & g_{z}(P_{z}w)^{-1}\cdot p_{\boldsymbol{i}_{k}(w)}r^{-k}\cdot g_{\alpha^{-k}z}(P_{\alpha^{-k}z}f_{\boldsymbol{i}_{k}(w)}^{-1}(w))\\
 & = & \frac{g_{\alpha^{-k}z}(P_{\alpha^{-k}z}T^{k}w)}{g_{z}(P_{z}w)}\cdot p_{\boldsymbol{i}_{k}(w)}r^{-k},
\end{eqnarray*}
which completes the proof of the lemma.
\end{proof}
Part (\ref{enu:dc for measues}) of Theorem \ref{thm:main_thm} follows
directly from part (\ref{enu:L^q dens in each dir}) combined with
the following proposition.
\begin{prop}
Let $1<q<\infty$, assume that $g_{z}\in L^{q}(\mathbb{R})$ for all
$z\in S$, and that $\arg\lambda\notin\pi\mathbb{Q}$. Then for every
$z\in S$ the measure $\nu_{z,w}$ has exact dimension $\dim_{H}\nu-1$
for $\nu$-a.e. $w\in\mathbb{C}$.
\end{prop}

\begin{proof}
By Proposition \ref{prop:bounded norms} there exists $C>0$ with
$\Vert g_{z}\Vert_{q}\le C$ for all $z\in S$. Fix $z\in S$, then
since $\mathcal{F}_{\lambda,\mathbf{a}}$ satisfies the SSC it suffices
to show that for $\nu$-a.e. $w\in K$,
\[
\underset{k}{\lim}\frac{\log\nu_{z,w}(K_{\boldsymbol{i}_{k}(w)})}{\log r^{k}}=\dim_{H}\nu-1\:.
\]

For every $k\ge1$ set $z_{k}=\alpha^{-k}z$ and
\[
A_{k}=\{w\in K\::\:g_{z_{k}}(P_{z_{k}}w)\ge k^{2/(q-1)}\}\:.
\]
We have,
\begin{multline*}
\nu(A_{k})=P_{z_{k}}\nu\{x\in\mathbb{R}\::\:\left(g_{z_{k}}(x)\right)^{q-1}\ge k^{2}\}\\
\le\frac{1}{k^{2}}\cdot\int\left(g_{z_{k}}(x)\right)^{q-1}\:dP_{z_{k}}\nu(x)=\frac{\Vert g_{z_{k}}\Vert_{q}^{q}}{k^{2}}\le\frac{C^{q}}{k^{2}}\:.
\end{multline*}
From this and since $(K,T,\nu)$ is measure preserving,
\[
\sum_{k=1}^{\infty}\nu(T^{-k}(A_{k}))=\sum_{k=1}^{\infty}\nu(A_{k})<\infty\:.
\]
Now from the Borel-Cantelli lemma it follows that for $\nu$-a.e.
$w\in K$,
\begin{equation}
g_{z_{k}}(P_{z_{k}}T^{k}w)\le k^{2/(q-1)}\text{ for all large enough \ensuremath{k\ge1}}\:.\label{eq:upper est for dens}
\end{equation}

For $k\ge1$ set
\[
B_{k}=\{w\in K\::\:g_{z_{k}}(P_{z_{k}}w)\le k^{-2}\},
\]
and let $J\subset\mathbb{R}$ be a bounded interval with $\mathrm{supp}(P_{\eta}\nu)\subset J$
for all $\eta\in S$. We have,
\begin{eqnarray*}
\nu(B_{k}) & = & P_{z_{k}}\nu\{x\in\mathbb{R}\::\:g_{z_{k}}(x)\le k^{-2}\}\\
 & = & \int_{J}1_{\{g_{z_{k}}(x)\le k^{-2}\}}g_{z_{k}}(x)\:dx\le\frac{\mathcal{L}(J)}{k^{2}},
\end{eqnarray*}
and so $\sum_{k=1}^{\infty}\nu(T^{-k}(B_{k}))<\infty$. From (\ref{eq:upper est for dens})
and the Borel-Cantelli lemma it now follows that for $\nu$-a.e. $w\in K$,
\begin{equation}
k^{-2}\le g_{z_{k}}(P_{z_{k}}T^{k}w)\le k^{2/(q-1)}\text{ for all large enough \ensuremath{k\ge1}}\:.\label{eq:est for dens}
\end{equation}

Denote by $H(p)$ the entropy of $p$, then $\dim_{H}\nu=\frac{H(p)}{-\log r}$
by (\ref{eq:Haus dim of ssm}). From the Shannon-McMillan-Breiman
theorem, applied to the ergodic system $(K,T,\nu)$ and partition
$\{f_{i}(K)\}_{i\in\Lambda}$, it follows that for $\nu$-a.e. $w\in K$,
\[
\underset{k}{\lim}\:-\frac{\log p_{\boldsymbol{i}_{k}(w)}}{k}=\underset{k}{\lim}\:-\frac{\log\nu(K_{\boldsymbol{i}_{k}(w)})}{k}=H(p)\:.
\]
From This, Lemma \ref{lem:slices by densities}, and (\ref{eq:est for dens}),
we get that for $\nu$-a.e. $w\in K$
\begin{multline*}
\underset{k}{\lim}\frac{\log\nu_{z,w}(K_{\boldsymbol{i}_{k}(w)})}{\log r^{k}}=\underset{k}{\lim}\:\frac{1}{k\log r}\cdot\log\left(\frac{g_{z_{k}}(P_{z_{k}}T^{k}w)}{g_{z}(P_{z}w)}\cdot p_{\boldsymbol{i}_{k}(w)}r^{-k}\right)\\
=\underset{k}{\lim}\frac{1}{k\log r}\cdot\log\left(p_{\boldsymbol{i}_{k}(w)}r^{-k}\right)=\frac{H(p)}{-\log r}-1=\dim_{H}\mu-1,
\end{multline*}
which completes the proof of the proposition.
\end{proof}

\section{\label{sec:Proof-of-Cor for sets}Proof of Corollary \ref{cor:cor from main thm}}
\begin{proof}[Proof of Corollary \ref{cor:cor from main thm}]
Let $\lambda\in\mathbb{D}\setminus\mathcal{E}$ be with $\arg\lambda\notin\pi\mathbb{Q}$,
$\Lambda$ a finite nonempty set, and $\mathbf{a}\in\mathbb{C}^{\Lambda}$
such that $\mathcal{F}_{\lambda,\mathbf{a}}$ satisfies the SSC. Set
$K=K_{\lambda,\mathbf{a}}$, let $s=\dim_{H}K$ and assume $s>1$.
Recall that for $z\in S$ and $x\in\mathbb{R}$ we write $K_{z,x}=K\cap P_{z}^{-1}\{x\}$.

Let $(p_{i})_{i\in\Lambda}=p\in\mathbb{P}_{\Lambda}$ be with $p_{i}=|\Lambda|^{-1}$
for each $i\in\Lambda$, and set $\nu=\nu_{\lambda,\mathbf{a}}^{p}$.
Note that $\nu=C_{0}\cdot\mathcal{H}^{s}|_{K}$ for some normalizing
constant $0<C_{0}<\infty$. By (\ref{eq:L^q dim of ssm}) and (\ref{eq:Haus dim of sss}),
\[
D_{2}(\nu)=\frac{\log\Vert p\Vert_{2}^{2}}{\log|\lambda|}=\frac{\log|\Lambda|^{-1}}{\log|\lambda|}=\dim_{H}K>1\:.
\]
Hence, by part (\ref{enu:L^q dens in each dir}) of Theorem \ref{thm:main_thm}
and Proposition \ref{prop:bounded norms}, it follows that $\{P_{z}\nu\}_{z\in S}$
is a Bounded subset of $\left(L^{2}(\mathbb{R}),\Vert\cdot\Vert_{2}\right)$.
Let $C>0$ be with $\Vert g_{z}\Vert_{2}\le C$ for all $z\in S$,
where $g_{z}=\frac{dP_{z}\nu}{d\mathcal{L}}$.

Fix $z\in S$ and set
\[
B=\{w\in K\::\:\dim_{H}K_{z,P_{z}w}=s-1\}\:.
\]
By part (\ref{enu:dc for measues}) of Theorem \ref{thm:main_thm},
\[
\dim_{H}\nu_{z,w}=\dim_{H}\nu-1=s-1\text{ for \ensuremath{\nu}-a.e. \ensuremath{w\in K}.}
\]
Additionally, $\nu_{z,w}$ is supported on $K_{z,P_{z}w}$ for $\nu$-a.e.
$w\in K$. This shows that $\dim_{H}K_{z,P_{z}w}\ge s-1$ for $\nu$-a.e.
$w\in K$. From Theorem 5.8 in \cite{Fal} it follows that $\dim_{H}(K_{z,x})\le s-1$
for $\mathcal{L}$-a.e. $x\in\mathbb{R}$. Since $P_{z}\nu\ll\mathcal{L}$
this shows that $\dim_{H}K_{z,P_{z}w}\le s-1$ for $\nu$-a.e. $w\in K$,
which gives $\nu(B)=1$.

Write
\[
A=\{w\in K\::\:\mathcal{P}^{s-1}(K_{z,P_{z}w})>0\},
\]
and recall the map $T:K\rightarrow K$ from the previous section.
Let $\alpha\in S$ and $0<r<1$ be with $\lambda=r\alpha$, and for
each $k\ge1$ set $z_{k}=\alpha^{-k}z$. Fix $n\ge1$, for each $k\ge1$
let
\[
A_{n,k}=\{w\in K\::\:g_{z_{k}}(P_{z_{k}}T^{k}w)\le n\},
\]
and write
\[
A_{n}=\cap_{N\ge1}\cup_{k\ge N}A_{n,k}.
\]
Since $(K,T,\nu)$ is measure preserving,
\begin{eqnarray*}
\nu(K\setminus A_{n,k}) & = & P_{z_{k}}\nu\{x\in\mathbb{R}\::\:g_{z_{k}}(x)>n\}\\
 & \le & \frac{1}{n}\int g_{z_{k}}(x)\:dP_{z_{k}}\nu(x)=\frac{\Vert g_{z_{k}}\Vert_{2}^{2}}{n}\le\frac{C^{2}}{n}\:.
\end{eqnarray*}
Hence,
\[
\nu(A_{n})=\underset{N}{\lim}\:\nu(\cup_{k\ge N}A_{n,k})\ge\underset{N}{\liminf}\:\nu(A_{n,N})\ge1-\frac{C^{2}}{n}\:.
\]
Write
\[
D=\{w\in K\::\:\nu_{z,w}(A_{n})>0\},
\]
then $\nu(D)\ge1-\frac{C^{2}}{n}$.

For $k\ge1$ and $w\in K$ let $K_{\boldsymbol{i}_{k}(w)}$ be as
defined in the previous section. From Lemma \ref{lem:slices by densities}
we get that for $\nu$-a.e. $w\in K$ and $\nu_{z,w}$-a.e. $\eta\in A_{n}$,
\begin{eqnarray*}
\underset{k\rightarrow\infty}{\liminf}\:\frac{\nu_{z,w}(K_{\boldsymbol{i}_{k}(\eta)})}{|\lambda|^{k(s-1)}} & = & \underset{k\rightarrow\infty}{\liminf}\:\frac{\nu_{z,\eta}(K_{\boldsymbol{i}_{k}(\eta)})}{|\lambda|^{k(s-1)}}\\
 & = & \underset{k\rightarrow\infty}{\liminf}\:\frac{g_{z_{k}}(P_{z_{k}}T^{k}\eta)\cdot|\Lambda|^{-k}|\lambda|^{-k}}{g_{z}(P_{z}\eta)|\lambda|^{k(s-1)}}\\
 & \le & \frac{n}{g_{z}(P_{z}\eta)}\underset{k\rightarrow\infty}{\lim}\:\frac{|\Lambda|^{-k}}{|\lambda|^{ks}}=\frac{n}{g_{z}(P_{z}w)},
\end{eqnarray*}
where in the last equality we have used $s=\frac{\log|\Lambda|}{-\log|\lambda|}$.
Now since $\mathcal{F}_{\lambda,\mathbf{a}}$ satisfies the SSC, we
get that for $\nu$-a.e. $w\in K$ there exists $0<C_{w,n}<\infty$
with,
\[
\underset{\delta\downarrow0}{\liminf}\:\frac{\nu_{z,w}(B(\eta,\delta))}{(2\delta)^{s-1}}\le C_{w,n}\text{ for }\nu_{z,w}\text{-a.e. }\eta\in A_{n}\:.
\]
From this and by Theorem 6.11 in \cite{M} it follows that for $\nu$-a.e.
$w\in D$,
\[
\mathcal{P}^{s-1}(K_{z,P_{z}w})\ge C_{w,n}^{-1}\cdot\nu_{z,w}(A_{n})>0,
\]
and so $w\in A$. This gives $\nu(A)\ge\nu(D)\ge1-\frac{C^{2}}{n}$,
which shows $\nu(A)=1$ since $n\ge1$ can be taken to be arbitrarily
large. We have thus shown $\nu(A\cap B)=1$. From this and since the
measures $\nu$ and $\mathcal{H}^{s}|_{K}$ are equivalent, we obtain
part (\ref{enu:H^s a.s. stat}) of the corollary.

Write,
\[
F=\{x\in\mathbb{R}\::\:\dim_{H}(K_{z,x})=s-1\text{ and }\mathcal{P}^{s-1}(K_{z,x})>0\}\:.
\]
From $\nu(A\cap B)=1$ we get $P_{z}\nu(F)=1$. Hence by the Cauchy-Schwarz
inequality,
\[
1=P_{z}\nu(F)=\int1_{F}(x)g_{z}(x)\:dx\le\Vert1_{F}\Vert_{2}\cdot\Vert g_{z}\Vert_{2}\le C\Vert1_{F}\Vert_{2},
\]
which gives $\mathcal{L}(F)\ge C^{-2}$. This completes the proof
of part (\ref{enu:uni bound proj}) of the corollary with $c=C^{-2}$.
\end{proof}

\section{\label{sec:Proof-of equiv measures}Proof of part (\ref{enu:equiv of measures})
of Theorem \ref{thm:main_thm}}

The following proof is an adaptation of the argument in \cite[Proposition 3.1]{PSS}.
\begin{proof}[Proof of part (\ref{enu:equiv of measures}) of Theorem \ref{thm:main_thm}]
Let $\lambda\in\mathbb{D}\setminus\mathcal{E}$ be with $\arg\lambda\notin\pi\mathbb{Q}$,
$\Lambda$ a finite nonempty set, $\mathbf{a}\in\mathbb{C}^{\Lambda}$
such that $\mathcal{F}_{\lambda,\mathbf{a}}$ satisfies the SSC, and
$p\in\mathbb{P}_{\Lambda}$ with $D_{q}(\nu_{\lambda,\mathbf{a}}^{p})>1$
for some $1<q<\infty$. Set $K=K_{\lambda,\mathbf{a}}$, $\nu=\nu_{\lambda,\mathbf{a}}^{p}$,
and $f_{i}=f_{\lambda,\mathbf{a}}^{i}$ for $i\in\Lambda$. By part
(\ref{enu:L^q dens in each dir}) of Theorem \ref{thm:main_thm} it
follows $P_{z}\nu\in L^{q}(\mathbb{R})$ for each $z\in S$. Hence
it suffices to show that $\mathcal{L}|_{P_{z}(K)}\ll P_{z}\nu$ for
each $z\in S$.

By Proposition \ref{prop:bounded norms} there exists $C>0$ with
$\Vert g_{z}\Vert_{q}\le C$ for all $z\in S$, where $g_{z}=\frac{dP_{z}\nu}{d\mathcal{L}}$.
From
\[
\dim_{H}K\ge\dim_{H}\nu\ge D_{q}(\nu)>1
\]
and part (\ref{enu:uni bound proj}) of Corollary \ref{cor:cor from main thm},
it follows that there exists $c>0$ with
\begin{equation}
\mathcal{L}(P_{z}(K))>c\text{ for all }z\in S\:.\label{eq:uni bd on len of proj}
\end{equation}

Write,
\[
\beta=\sup\{\frac{\mathcal{L}(A)}{\mathcal{L}(P_{z}(K))}\::\:z\in S,\:A\subset P_{z}(K)\text{ is Borel, and }P_{z}\nu(A)=0\}\:.
\]
Assume by contradiction that $\beta=1$. Given $\epsilon>0$ there
exists $z\in S$ and a Borel set $A\subset P_{z}(K)$ with $P_{z}\nu(A)=0$
and $\frac{\mathcal{L}(A)}{\mathcal{L}(P_{z}(K))}>1-\epsilon$. We
have
\[
\mathcal{L}(P_{z}(K)\setminus A)<\epsilon\cdot\mathcal{L}(P_{z}(K)),
\]
hence by Hölder's inequality,
\begin{eqnarray*}
1 & = & P_{z}\nu(P_{z}(K)\setminus A)=\int_{P_{z}(K)\setminus A}g_{z}\:d\mathcal{L}\\
 & \le & \Vert g_{z}\Vert_{q}\cdot(\mathcal{L}(P_{z}(K)\setminus A))^{1-q^{-1}}<C\cdot(\epsilon\cdot\mathcal{L}(P_{z}(K)))^{1-q^{-1}}\:.
\end{eqnarray*}
For a sufficiently small $\epsilon>0$ this clearly yields a contradiction,
and so we must have $\beta<1$.

Fix $z\in S$ and let $A_{0}\subset P_{z}(K)$ be a Borel set with
$P_{z}\nu(A_{0})=0$. It suffices to prove that $\mathcal{L}(A_{0})=0$.
We will show that for each $x\in\mathbb{R}$
\begin{equation}
\underset{\delta\downarrow0}{\limsup}\:\frac{\mathcal{L}((x-\delta,x+\delta)\cap A_{0})}{2\delta}<1,\label{eq:density of A_0}
\end{equation}
and so that $A_{0}$ has no Lebesgue density points. From this and
Lebesgue's density theorem it follows that we must have $\mathcal{L}(A_{0})=0$.
If $x\notin P_{z}(K)$ then $(x-\delta,x+\delta)\cap A_{0}=\emptyset$
for all sufficiently small $\delta>0$, and so (\ref{eq:density of A_0})
is clearly satisfied.

Fix $x\in P_{z}(K)$ and $0<\delta<\frac{c}{2}$. By (\ref{eq:uni bd on len of proj})
there exists $m\ge1$ and $i_{1}...i_{m}=\boldsymbol{i}\in\Lambda^{m}$
such that $x\in P_{z}f_{\boldsymbol{i}}(K)\subset(x-\delta,x+\delta)$
and 
\begin{equation}
P_{z}f_{i_{1}...i_{m-1}}(K)\nsubseteq(x-\delta,x+\delta),\label{eq:proj cyl not contained}
\end{equation}
where recall that $f_{\boldsymbol{i}}=f_{i_{1}}\circ...\circ f_{i_{m}}$.
Let $\alpha\in S$ and $0<r<1$ be with $\lambda=r\alpha$ and set
$a=-P_{z}f_{\boldsymbol{i}}(0)$. Then
\begin{equation}
P_{z}f_{\boldsymbol{i}}=\tau_{a}M_{r^{m}}P_{\alpha^{-m}z},\label{eq:proj of map}
\end{equation}
where recall that $\tau_{a}(y)=y-a$ and $M_{r^{m}}(y)=r^{m}y$ for
$y\in\mathbb{R}$. By (\ref{eq:proj cyl not contained}),
\[
r^{m}\cdot diam(K)=r\cdot diam(f_{i_{1}...i_{m-1}}(K))\ge r\delta\:.
\]
From this, (\ref{eq:proj of map}), and (\ref{eq:uni bd on len of proj}),
\begin{equation}
\mathcal{L}(P_{z}f_{\boldsymbol{i}}(K))=r^{m}\cdot\mathcal{L}(P_{\alpha^{-m}z}(K))>\frac{r\delta c}{diam(K)}=b\delta,\label{eq:est of proj cyl}
\end{equation}
with $b=\frac{rc}{diam(K)}$.

Since $\nu$ is self-similar we have $P_{z}f_{\boldsymbol{i}}\nu\ll P_{z}\nu$.
From this and since $P_{z}\nu(A_{0}\cap P_{z}f_{\boldsymbol{i}}(K))=0$,
\begin{multline*}
0=P_{z}f_{\boldsymbol{i}}\nu(A_{0}\cap P_{z}f_{\boldsymbol{i}}(K))=\tau_{a}M_{r^{m}}P_{\alpha^{-m}z}\nu(A_{0}\cap\tau_{a}M_{r^{m}}P_{\alpha^{-m}z}(K))\\
=P_{\alpha^{-m}z}\nu(M_{r^{-m}}\tau_{-a}(A_{0})\cap P_{\alpha^{-m}z}(K))\:.
\end{multline*}
Now by the definition of $\beta$,
\[
\mathcal{L}(M_{r^{-m}}\tau_{-a}(A_{0})\cap P_{\alpha^{-m}z}(K))\le\beta\cdot\mathcal{L}(P_{\alpha^{-m}z}(K))\:.
\]
Hence,
\[
\mathcal{L}(A_{0}\cap\tau_{a}M_{r^{m}}P_{\alpha^{-m}z}(K))\le\beta\cdot\mathcal{L}(\tau_{a}M_{r^{m}}P_{\alpha^{-m}z}(K)),
\]
which gives,
\[
\mathcal{L}(A_{0}\cap P_{z}f_{\boldsymbol{i}}(K))\le\beta\cdot\mathcal{L}(P_{z}f_{\boldsymbol{i}}(K))\:.
\]
From this and (\ref{eq:est of proj cyl}) we get,
\begin{eqnarray*}
\mathcal{L}((x-\delta,x+\delta)\setminus A_{0}) & \ge & \mathcal{L}(P_{z}f_{\boldsymbol{i}}(K)\setminus A_{0})\\
 & = & \mathcal{L}(P_{z}f_{\boldsymbol{i}}(K))-\mathcal{L}(P_{z}f_{\boldsymbol{i}}(K)\cap A_{0})\\
 & \ge & (1-\beta)\mathcal{L}(P_{z}f_{\boldsymbol{i}}(K))\ge(1-\beta)b\delta\:.
\end{eqnarray*}
Since $(1-\beta)b>0$ does not depend on $\delta$ this gives (\ref{eq:density of A_0}),
which completes the proof.
\end{proof}

$\newline$\textsc{Centre for Mathematical Sciences,\newline Wilberforce Road, Cambridge CB3 0WA, UK}$\newline$$\newline$\textit{E-mail: }
\texttt{ariel.rapaport@mail.huji.ac.il}
\end{document}